\documentclass[12pt]{amsart}
\usepackage{amsthm}
\usepackage{amsmath}
\usepackage{amssymb}
\usepackage{verbatim}
\usepackage[mathscr]{eucal}
\allowdisplaybreaks[4]
\begin{document}

\def\alp{\alpha}
\def\bet{\beta}
\def\gam{\gamma}
\def\del{\delta}
\def\eps{\epsilon}
\def\zet{\zeta}
\def\tht{\theta}
\def\iot{\iota}
\def\kap{\kappa}
\def\lam{\lambda}
\def\sig{\sigma}
\def\ups{\upsilon}
\def\ome{\omega}
\def\vep{\varepsilon}
\def\vth{\vartheta}
\def\vpi{\varpi}
\def\vrh{\varrho}
\def\vsi{\varsigma}
\def\vph{\varphi}
\def\Gam{\Gamma}
\def\Del{\Delta}
\def\Tht{\Theta}
\def\Lam{\Lambda}
\def\Sig{\Sigma}
\def\Ups{\Upsilon}
\def\Ome{\Omega}
\def\vTh{\varTheta}
\def\vGm{\varGamma}
\def\vPh{\varPhi}
\def\vPs{\varPsi}

\def\frka{{\mathfrak a}}    \def\frkA{{\mathfrak A}}
\def\frkb{{\mathfrak b}}    \def\frkB{{\mathfrak B}}
\def\frkc{{\mathfrak c}}    \def\frkC{{\mathfrak C}}
\def\frkd{{\mathfrak d}}    \def\frkD{{\mathfrak D}}
\def\frke{{\mathfrak e}}    \def\frkE{{\mathfrak E}}
\def\frkf{{\mathfrak f}}    \def\frkF{{\mathfrak F}}
\def\frkg{{\mathfrak g}}    \def\frkG{{\mathfrak G}}
\def\frkh{{\mathfrak h}}    \def\frkH{{\mathfrak H}}
\def\frki{{\mathfrak i}}    \def\frkI{{\mathfrak I}}
\def\frkj{{\mathfrak j}}    \def\frkJ{{\mathfrak J}}
\def\frkk{{\mathfrak k}}    \def\frkK{{\mathfrak K}}
\def\frkl{{\mathfrak l}}    \def\frkL{{\mathfrak L}}
\def\frkm{{\mathfrak m}}    \def\frkM{{\mathfrak M}}
\def\frkn{{\mathfrak n}}    \def\frkN{{\mathfrak N}}
\def\frko{{\mathfrak o}}    \def\frkO{{\mathfrak O}}
\def\frkp{{\mathfrak p}}    \def\frkP{{\mathfrak P}}
\def\frkq{{\mathfrak q}}    \def\frkQ{{\mathfrak Q}}
\def\frkr{{\mathfrak r}}    \def\frkR{{\mathfrak R}}
\def\frks{{\mathfrak s}}    \def\frkS{{\mathfrak S}}
\def\frkt{{\mathfrak t}}    \def\frkT{{\mathfrak T}}
\def\frku{{\mathfrak u}}    \def\frkU{{\mathfrak U}}
\def\frkv{{\mathfrak v}}    \def\frkV{{\mathfrak V}}
\def\frkw{{\mathfrak w}}    \def\frkW{{\mathfrak W}}
\def\frkx{{\mathfrak x}}    \def\frkX{{\mathfrak X}}
\def\frky{{\mathfrak y}}    \def\frkY{{\mathfrak Y}}
\def\frkz{{\mathfrak z}}    \def\frkZ{{\mathfrak Z}}

\def\cal{\fam2}
\def\cala{{\cal A}}
\def\calb{{\cal B}}
\def\calc{{\cal C}}
\def\cald{{\cal D}}
\def\cale{{\cal E}}
\def\calf{{\cal F}}
\def\calg{{\cal G}}
\def\calh{{\cal H}}
\def\cali{{\cal I}}
\def\calj{{\cal J}}
\def\calk{{\cal K}}
\def\call{{\cal L}}
\def\calm{{\cal M}}
\def\caln{{\cal N}}
\def\calo{{\cal O}}
\def\calp{{\cal P}}
\def\calq{{\cal Q}}
\def\calr{{\cal R}}
\def\cals{{\cal S}}
\def\calt{{\cal T}}
\def\calu{{\cal U}}
\def\calv{{\cal V}}
\def\calw{{\cal W}}
\def\calx{{\cal X}}
\def\caly{{\cal Y}}
\def\calz{{\cal Z}}

\def\AA{{\mathbb A}}
\def\BB{{\mathbb B}}
\def\CC{{\mathbb C}}
\def\DD{{\mathbb D}}
\def\EE{{\mathbb E}}
\def\FF{{\mathbb F}}
\def\GG{{\mathbb G}}
\def\HH{{\mathbb H}}
\def\II{{\mathbb I}}
\def\JJ{{\mathbb J}}
\def\KK{{\mathbb K}}
\def\LL{{\mathbb L}}
\def\MM{{\mathbb M}}
\def\NN{{\mathbb N}}
\def\OO{{\mathbb O}}
\def\PP{{\mathbb P}}
\def\QQ{{\mathbb Q}}
\def\RR{{\mathbb R}}
\def\SS{{\mathbb S}}
\def\TT{{\mathbb T}}
\def\UU{{\mathbb U}}
\def\VV{{\mathbb V}}
\def\WW{{\mathbb W}}
\def\XX{{\mathbb X}}
\def\YY{{\mathbb Y}}
\def\ZZ{{\mathbb Z}}

\def\bfa{{\mathbf a}}    \def\bfA{{\mathbf A}}
\def\bfb{{\mathbf b}}    \def\bfB{{\mathbf B}}
\def\bfc{{\mathbf c}}    \def\bfC{{\mathbf C}}
\def\bfd{{\mathbf d}}    \def\bfD{{\mathbf D}}
\def\bfe{{\mathbf e}}    \def\bfE{{\mathbf E}}
\def\bff{{\mathbf f}}    \def\bfF{{\mathbf F}}
\def\bfg{{\mathbf g}}    \def\bfG{{\mathbf G}}
\def\bfh{{\mathbf h}}    \def\bfH{{\mathbf H}}
\def\bfi{{\mathbf i}}    \def\bfI{{\mathbf I}}
\def\bfj{{\mathbf j}}    \def\bfJ{{\mathbf J}}
\def\bfk{{\mathbf k}}    \def\bfK{{\mathbf K}}
\def\bfl{{\mathbf l}}    \def\bfL{{\mathbf L}}
\def\bfm{{\mathbf m}}    \def\bfM{{\mathbf M}}
\def\bfn{{\mathbf n}}    \def\bfN{{\mathbf N}}
\def\bfo{{\mathbf o}}    \def\bfO{{\mathbf O}}
\def\bfp{{\mathbf p}}    \def\bfP{{\mathbf P}}
\def\bfq{{\mathbf q}}    \def\bfQ{{\mathbf Q}}
\def\bfr{{\mathbf r}}    \def\bfR{{\mathbf R}}
\def\bfs{{\mathbf s}}    \def\bfS{{\mathbf S}}
\def\bft{{\mathbf t}}    \def\bfT{{\mathbf T}}
\def\bfu{{\mathbf u}}    \def\bfU{{\mathbf U}}
\def\bfv{{\mathbf v}}    \def\bfV{{\mathbf V}}
\def\bfw{{\mathbf w}}    \def\bfW{{\mathbf W}}
\def\bfx{{\mathbf x}}    \def\bfX{{\mathbf X}}
\def\bfy{{\mathbf y}}    \def\bfY{{\mathbf Y}}
\def\bfz{{\mathbf z}}    \def\bfZ{{\mathbf Z}}

\def\scra{{\mathscr A}}
\def\scrb{{\mathscr B}}
\def\scrc{{\mathscr C}}
\def\scrd{{\mathscr D}}
\def\scre{{\mathscr E}}
\def\scrf{{\mathscr F}}
\def\scrg{{\mathscr G}}
\def\scrh{{\mathscr H}}
\def\scri{{\mathscr I}}
\def\scrj{{\mathscr J}}
\def\scrk{{\mathscr K}}
\def\scrl{{\mathscr L}}
\def\scrm{{\mathscr M}}
\def\scrn{{\mathscr N}}
\def\scro{{\mathscr O}}
\def\scrp{{\mathscr P}}
\def\scrq{{\mathscr Q}}
\def\scrr{{\mathscr R}}
\def\scrs{{\mathscr S}}
\def\scrt{{\mathscr T}}
\def\scru{{\mathscr U}}
\def\scrv{{\mathscr V}}
\def\scrw{{\mathscr W}}
\def\scrx{{\mathscr X}}
\def\scry{{\mathscr Y}}
\def\scrz{{\mathscr Z}}

\def\phm{\phantom}
\def\ds{\displaystyle }
\def\smallstrut{\vphantom{\vrule height 3pt }}
\def\bdm #1#2#3#4{\left(
\begin{array} {c|c}{\ds{#1}}
 & {\ds{#2}} \\ \hline
{\ds{#3}\vphantom{\ds{#3}^1}} &  {\ds{#4}}
\end{array}
\right)}

\def\GL{\mathrm{GL}}
\def\SL{\mathrm{SL}}
\def\Sp{\mathrm{Sp}}
\def\SU{\mathrm{SU}}
\def\SO{\mathrm{SO}}
\def\GSp{\mathrm{GSp}}
\def\Spin{\mathrm{Spin}}
\def\U{\mathrm{U}}
\def\O{\mathrm{O}}
\def\Mat{\mathrm{M}}
\def\Nr{\mathrm{Nr}}
\def\Tr{\mathrm{Tr}}
\def\tr{\mathrm{tr}}
\def\Nrm{\mathrm{Nrm}}
\def\Paf{\mathrm{Paf}}
\def\Prm{\mathrm{Prm}}
\def\Lif{\mathrm{Lift}}
\def\lif{\mathrm{lift}}
\def\rank{\mathrm{rank}}
\def\cusp{\mathrm{cusp}}
\def\new{\mathrm{new}}
\def\oname{\mathrm}
\def\Re{\mathrm{Re}}
\def\Imm{\mathrm{Im}}
\def\sgn{\mathrm{sgn}}
\def\diag{\mathrm{diag}}
\def\La{\langle}
\def\Ra{\rangle}

\def\scre{{\mathscr E}}
\def\scrh{{\mathscr H}}
\def\scrv{{\mathscr V}}

\def\trs{\,^t\!}
\def\iu{\sqrt{-1}}
\def\oo{\hbox{\bf 0}}
\def\ono{\hbox{\bf 1}}
\def\smallcirc{\lower .3em \hbox{\rm\char'27}\!}
\def\AAf{\AA_{\rm f}}
\def\thalf{\textstyle{\frac{1}{2}}}
\def\bsl{\backslash}
\def\wtl{\widetilde}
\def\tx{\text}
\def\Chi{\underline{\chi}}
\def\Psi{\underline{\psi}}
\def\oln{\overline}
\def\til{\tilde}
\def\Ind{\operatorname{Ind}}
\def\ord{\operatorname{ord}}
\def\pal{\partial}
\def\beq{\begin{equation}}
\def\eeq{\end{equation}}
\def\Let{\begin{pmatrix}}
\def\Rit{\end{pmatrix}}
\def\Left{\left(\begin{smallmatrix}}
\def\Right{\end{smallmatrix}\right)}
\def\Leftt{\begin{smallmatrix}}
\def\Rightt{\end{smallmatrix}}
\def\Sum{\displaystyle \sum}
\def\Prod{\displaystyle \prod}
\def\Iot{^{\iota}}
\def\Pr{^{\prime}}
\def\I{\textgt{i}}
\def\J{\textgt{j}}
\def\K{\textgt{k}}

\newcounter{one}
\setcounter{one}{1}
\newcounter{two}
\setcounter{two}{2}
\newcounter{thr}
\setcounter{thr}{3}
\newcounter{fou}
\setcounter{fou}{4}
\newcounter{fiv}
\setcounter{fiv}{5}
\newcounter{six}
\setcounter{six}{6}

\def\FRAC#1#2{\leavevmode\kern.1em
\raise.5ex\hbox{\the\scriptfont0 #1}\kern-.1em
/\kern-.15em\lower.25ex\hbox{\the\scriptfont0 #2}}
\newcommand{\Leg}[2]{\left(\!\FRAC{\ensuremath{#1}}{\ensuremath{#2}}\right)}
\def\DFRAC#1#2{\leavevmode\kern.1em
\raise1.0ex\hbox{#1}\kern-.2em
{\Big/}\kern-.3em\lower.9ex\hbox{#2}}
\newcommand{\DLeg}[2]{\left(\!\DFRAC{\ensuremath{#1}}{\ensuremath{#2}}\right)}
\def\mtrx#1#2#3#4{\left(\begin{smallmatrix} {\text{\small $#1$}}\; & {\text{\small $#2$}} \\ \vphantom{C^{C^C}}{\text{\small $#3$}}\; & \vphantom{C^2}{\text{\small $#4$}} \end{smallmatrix}\right)}

\newcommand{\sfrac}[2]{%
\raisebox{0.2em}{$#1$}\kern-0.1em%
\raisebox{0.1em}{/}\kern-0.1em\raisebox{-0.2em}{$#2$}}

\def\lddots{\mathinner{\mskip1mu\raise1pt\vbox{\kern7pt\hbox{.}}\mskip2mu\raise4pt\hbox{.}\mskip2mu\raise7pt\hbox{.}\mskip1mu}}
\def\onon{\Left 0 & -1 \\ 1 & 0\Right}

\def\today{\ifcase\month\or
 January\or February\or March\or April\or May\or June\or
 July\or August\or September\or October\or November\or December\fi
 \space\number\day, \number\year}

\makeatletter
\@addtoreset{equation}{section}
\def\theequation{\thesection.\arabic{equation}}

\theoremstyle{plain}
\newtheorem{theorem}{Theorem}[section]
\newtheorem*{main_theorem}{Theorem}
\newtheorem*{main_remark}{Remark}
\newtheorem{lemma}[theorem]{Lemma}
\newtheorem{proposition}[theorem]{Proposition}
\theoremstyle{definition}
\newtheorem{definition}[theorem]{Definition}
\newtheorem{conjecture}[theorem]{Conjecture}
\newtheorem*{main_conjecture}{Conjecture \ref{conj:91}}
\theoremstyle{remark}
\newtheorem{remark}[theorem]{Remark}
\newtheorem{corollary}[theorem]{Corollary}


\title[]{Maass relations in higher genus}
\author{Shunsuke Yamana}
\address{Graduate school of mathematics, Kyoto University, Kitashirakawa, Kyoto, 606-8502, Japan}
\email{yamana07@math.kyoto-u.ac.jp}
\begin{abstract}
For an arbitrary even genus $2n$ we show that the subspace of Siegel cusp forms of degree $2n$ generated by Ikeda lifts of elliptic cusp forms can be characterized by certain linear relations among Fourier coefficients. 
This generalizes the classical Maass relations in genus two to higher degrees. 
\end{abstract}
\thanks{The author should like to express his sincere thanks to Prof.~Ikeda for suggesting the problem, constant help throughout this paper, encouragement and patience. He is supported by JSPS Research Fellowships for Young Scientists. }
\maketitle


\section*{\bf Introduction}\label{intro}

The purpose of this paper is to give a characterization of the image of Ikeda's lifting by certain linear relations among Fourier coefficients. 

Let us describe our results. 
Let us fix a positive integer $n$. 
We denote by $T^+_m$ the set of positive definite symmetric half-integral matrices of size $m$. 
Put $D_h=\det(2h)$ for $h\in T^+_{2n}$. 
Let $\frkd_h$ be the absolute value of discriminant of $\QQ(((-1)^nD_h)^{1/2})/\QQ$. 
Put $\frkf_h=(\frkd_h^{-1}D_h)^{1/2}$. 

Fix a prime number $p$. The Siegel series attached to $h$ is defined by 
\[b_p(h,s)=\sum_\alp\bfe_p(-\tr(h\alp))\nu(\alp)^{-s}, \]
where $\alp$ extends over all symmetric matrices of rank $2n$ with entries in $\QQ_p/\ZZ_p$. 
Here, we put $\bfe_p(x)=e^{-2\pi\iu x\Pr}$, where $x\Pr$ is the fractional part of $x\in\QQ_p$, and $\nu(\alp)=[\alp\ZZ^{2n}+\ZZ^{2n}:\ZZ^{2n}]$. 

As is well-known, $b_p(h,s)$ is a product of two polynomials $\gam_{p,h}(p^{-s})$ and $F_{p,h}(p^{-s})$, where
\[\gam_{p,h}(X)=(1-X)\biggl(1-\biggl(\frac{(-1)^n\frkd_h}{p}\biggl)p^nX\biggl)^{-1}\prod_{j=1}^n(1-p^{2j}X^2)\] and the constant term of $F_{p,h}$ is $1$. Put 
\[\til{F}_{p,h}(X)=X^{-\ord_p\frkf_h}F_{p,h}(p^{-n-1/2}X). \]
Then the following functional equation holds
\[\til{F}_{p,h}(X)=\til{F}_{p,h}(X^{-1}). \]

Ikeda constructed a lifting from $S_{2k}(\SL_2(\ZZ))$ to $S_{k+n}(\Sp_{2n}(\ZZ))$ for each positive integer $k$ such that $k\equiv n\pmod 2$. 
\begin{main_theorem}[Ikeda {\cite{Ik}}]
Let $f\in S_{2k}(\SL_2(\ZZ))$ with $k\equiv n\pmod 2$ be a normalized Hecke eigenform, the $L$-function of which is given by
\[\prod_p(1-\alp_pp^{k-1/2-s})^{-1}(1-\alp_p^{-1}p^{k-1/2-s})^{-1}. \]
Let $g$ be a corresponding cusp form in the Kohnen plus space $S^+_{k+1/2}(4)$ under the Shimura correspondence. 
We define the function $F$ on the upper half-space $\frkH_n$ by  
\[F(Z)=\sum_{h\in T_{2n}^+}c_F(h)e^{2\pi\iu\tr(hZ)}, \]
where  
\[c_F(h)=c_g(\frkd_h)\frkf_h^{k-1/2}\prod_p\til{F}_{p,h}(\alp_p). \]
Here, we denote the $m$-th Fourier coefficient of $g$ by $c_g(m)$. 
Then $F$ is a cuspidal Hecke eigenform in $S_{k+n}(\Sp_{2n}(\ZZ))$. 
\end{main_theorem}

In \cite{Ko} Kohnen defined an integer $\phi(d; h)$ for each $h\in T_{2n}^+$ and each positive integer $d$ such that $\frkf_h$ is divisible by $d$. He showed that  
\begin{align*}
&I_{n,k}:&\sum_mc(m)e^{2\pi\iu m}&\mapsto \sum_{h\in T_{2n}^+}\sum_{d|\frkf_h}d^{k-1}\phi(d;h)c(d^{-2}D_h)e^{2\pi\iu\tr(hZ)}
\end{align*}
is a linear map from $S^+_{k+1/2}(4)$ to $S_{k+n}(\Sp_{2n}(\ZZ))$, which on Hecke eigenforms coincides with the Ikeda lifting. 

He also conjectured that if $F\in S_{k+n}(\Sp_{2n}(\ZZ))$ has a Fourier expansion of the form
\[F(Z)=\sum_{h\in T_{2n}^+}\sum_{d|\frkf_h}d^{k-1}\phi(d;h)c(d^{-2}D_h)e^{2\pi\iu\tr(hZ)} \]
for some function $c:\{m\in\NN\;|\;(-1)^km\equiv0,1\pmod 4\}\to \CC$ , then $F$ belongs to the image of the map $I_{n,k}$. 
If $n=1$, then this conjecture comes down to saying that the Maass space coincides with the image of the Saito-Kurokawa lifting, and hence it is true.   

Kohnen and Kojima proved the conjecture for all $n$ with $n\equiv 0,1\pmod 4$ in \cite{KK}. 
More precisely, they showed that  
\[g(\tau)=\sum_mc(m)e^{2\pi\iu m}\in S_{k+1/2}^+(4). \]

Let $S$ be a positive definite symmetric even-integral matrices of size $2n-1$. We write $J^\cusp_{k+n,S}$ for the space of Jacobi cusp forms of weight $k+n$ with index $S$ (see \S \ref{sec:3} for definition). 
We write the $S/2$-th Fourier-Jacobi coefficient of $F$ by $F_{S/2}$. 
The map $I_{k,n,S}$ is defined by 
\begin{align*}
&I_{n,k,S}:&S^+_{k+1/2}(4)&\to J^\cusp_{k+n,S}, & g&\mapsto I_{n,k}(g)_{S/2}. 
\end{align*}

We can take $S$ with $\det S=2$ if $n\equiv 0,1\pmod 4$. 
Then the map $I_{n,k,S}$ is an isomorphism and is given by 
\begin{align*}
&I_{n,k,S}:&\sum_mc(m)e^{2\pi\iu m}&\mapsto \sum_{(a,\alp)}c\Big(\det\mtrx{S}{S\alp}{\trs\alp S}{2a}\Big)e^{2\pi\iu(a\tau+S(\alp,w))}, 
\end{align*}
where $(a,\alp)$ extends over all pairs $(a, \alp)\in \ZZ\times S^{-1}(\ZZ^{2n-1})$ such that $a>S[\alp]/2$. 
This is a key observation in \cite{KK} and proves the conjecture when $n\equiv 0,1\pmod 4$. 

As is well-known, there is no $B$ with $\det(2B)=2$ if $n\equiv 2,3\pmod 4$. 
However, we shall show the following: 
\begin{main_theorem}[cf. Theorem \ref{thm:12}]
Under the notation above, there is $g\in S^+_{k+1/2}(4)$ with the following properties. 
\begin{enumerate}
\renewcommand\labelenumi{(\theenumi)}
\item If $n\equiv 3\pmod 4$, then  $c(m)=c_g(m)$. 
\item If $n\equiv 2\pmod 4$, then the following assertions holds:  
\begin{enumerate}
\renewcommand\labelenumi{(\theenumi)}
\item $F=I_{n,k}(g)$; 
\item if $m$ is not a square, then $c(m)=c_g(m)$; 
\item $c(f^2)=c_g(f^2)+(c(1)-c_g(1))f^k$ for every positive integer $f$. 
\end{enumerate}
\end{enumerate}
\end{main_theorem}


We now give a sketch of our proof of Theorem \ref{thm:12}. The space $J^{\cusp,M}_{\kap,S}$ consists of $\vPh\in J^{\cusp,M}_{\kap,S}$ which has a Fourier expansion of the form
\[\vPh(\tau,w)=\sum_{(a,\alp)}C\Big(\det\mtrx{S}{S\alp}{\trs\alp S}{2a}\Big)e^{2\pi\iu(a\tau+S(\alp,w))}\]
for some function $C$. 
Suppose that $S$ is maximal. 
The main result in \cite{Y} states that the map $I_{n,k,S}$ can be extended to a canonical isomorphism from a certain space of modular forms of weight $k+1/2$ onto $J^{\cusp,M}_{k+n, S}$ (we recall a precise formulation in \S \ref{sec:3}). 

The difficulty in attacking the case when $n\equiv 2,3\pmod 4$ is that $I_{n,k,S}(S^+_{k+1/2}(4))$ is a proper subspace of $J^{\cusp,M}_{k+n,S}$ if $\det S>2$. 
To solve the problem, we need the newform theory for modular forms of half-integral weight, which we recall in \S \ref{sec:2}. 
In \S \ref{sec:4} we see that the existence of a function $c$ gives rise to $g\in S^+_{k+1/2}(4)$ such that 
\[F_{S/2}=I_{n,k,S}(g)\] 
for all maximal $S$.  
This proves the conjecture for all $n$, and one could view the above condition as a natural generalization of Maass relations to higher genus. 


\section*{\bf Notation}\label{notation}

We denote by $S_\kap(\Sp_m(\ZZ))$ the space of Siegel cusp forms of weight $\kap$ with respect to the Siegel modular group $\Sp_m(\ZZ)$. 
Let $T^+_m$ be the set of positive definite symmetric half-integral matrices of size $m$. 
We denote the $h$-th Fourier coefficient by $c_F(h)$ for $F\in S_\kap(\Sp_m(\ZZ))$ and $h\in T^+_m$. 
Put $\bfe(z)=e^{2\pi\iu z}$ for a complex number $z$. 
Put $q=\bfe(\tau)$ for $\tau\in\frkH=\frkH_1$. 
Put
\[\frkD_k=\{m\in\NN\;|\;(-1)^km\equiv 0,1\pmod 4\}. \]

We fix a positive integer $n$ throughout this paper. 
For a non-zero element $N\in\QQ$, we denote the absolute value of the discriminant of $\QQ(((-1)^nN)^{1/2})/\QQ$ by $\frkd_N$. Put $\frkf_N=(\frkd_N^{-1}N)^{1/2}$ and $\frkf_p(N)=\ord_p\frkf_N$. 
For a rational prime $p$, we put $\Psi_p(N)=1$, $-1$, $0$ accordingly as $\QQ_p(\sqrt{N})$ is $\QQ_p$, an unramified quadratic extension or a ramified quadratic extension. 
Put $D_h=4^{[m/2]}\det h$ for $h\in T_m^+$, where $[\;]$ is the Gauss bracket.  
To simplify notation, we abbreviate $\frkd_h=\frkd_{D_h}$ and $\frkf_h=\frkf_{D_h}$. 
Put $h[A]=\trs AhA$ and $h(A,B)=\trs AhB$ for a matrices $A$ and $B$ if they are well-defined. 


\section{\bf Main Theorem}\label{sec:1}

For an integer $e$, we define $l_e\in \CC[X, X^{-1}]$ by 
\[l_e(X)=\begin{cases}
\dfrac{X^{e+1}-X^{-e-1}}{X-X^{-1}} &\tx{if $e\geq 0$. }\\
0 &\tx{if $e<0$. }
\end{cases}\]
For each prime number $p$, we put 
\[\lam_{p,N}=l_{\frkf_p(N)}-\Psi_p((-1)^nN)p^{-1/2}l_{\frkf_p(N)-1}. \]
(see Notation).  
Kohnen defined an integer $\phi(d;h)$ for each $h\in T_{2n}^+$ and each positive divisor $d$ of $\frkf_h$ (see \cite{Ko} or \cite{KK} for the precise definition). 
Let us note that the proof of the main result in \cite{Ko} states that 
\beq
\frkf_h^{k-1/2}\prod_p\wtl{F}_{p,h}(X_p)
=\sum_{d|\frkf_h}d^{k-1}\phi(d;h)(d^{-1}\frkf_h)^{k-1/2}\prod_p\lam_{p,d^{-2}D_h}(X_p)\label{tag:11}
\eeq
and the following theorem is a consequence of Ikeda's lifting, (\ref{tag:42}) and this formula. 

\begin{theorem}[Kohnen {\cite{Ko}}]\label{thm:11}
Under the notation of Ikeda's theorem, 
\beq
c_F(h)=\sum_{d|\frkf_h}d^{k-1}\phi(d;h)c_g(d^{-2}D_h). \label{tag:12}
\eeq
\end{theorem}

We define the linear map $I_{n,k}:S^+_{k+1/2}(4)\to S_{k+n}(\Sp_{2n}(\ZZ))$ by (\ref{tag:12}). 

\begin{definition}\label{def:11}
If $k\equiv n\pmod 2$, the space $S^M_{k+n}(\Sp_{2n}(\ZZ))$ is defined in the following way: 
$F\in S_{k+n}(\Sp_{2n}(\ZZ))$ is an element of $S^M_{k+n}(\Sp_{2n}(\ZZ))$ if there is a function $c:\frkD_k\to\CC$ such that all $h\in T_{2n}^+$ satisfy 
\beq
c_F(h)=\sum_{d|\frkf_h}d^{k-1}\phi(d;h)c(d^{-2}D_h). \label{tag:13}
\eeq
\end{definition}

\begin{lemma}\label{lem:11}
If $n\equiv2\pmod 4$ and $h$ is an element of $T^+_{2n}$ with $\frkd_h=1$, then we have 
\[\sum_{d|\frkf_h}d^{k-1}\phi(d;h)(d^{-1}\frkf_h)^k=0. \]
\end{lemma}

\begin{remark}\label{rem:11}
Lemma \ref{lem:11} shows that the parameter $c$ is not uniquely determined by $F$ in (\ref{tag:13}) if $n\equiv 2\pmod 4$. 
More precisely, a function $c\Pr:\frkD_k\to\CC$ defined by 
\[c\Pr(m)=\begin{cases}
c(m) &\tx{if $m\notin \QQ^{\times2}$ }\\
c(m)+a\cdot \frkf_m^k &\tx{if $m\in \QQ^{\times2}$ }
\end{cases}\]
is also a parameter of $F$ for each complex number $a$. We shall show that these are all parameters of $F$ (cf. Corollary \ref{cor:41}).  
\end{remark}

\begin{proof}
Although $\frkD_k$ contains $1$, it is well-known that there is no $B\in T^+_{2n}$ such that $D_B=1$. 
Therefore either \cite[Proposition 1]{Ko} or \cite[Lemma 15.2]{Ik} shows that $\prod_p\wtl{F}_{p,h}(p^{1/2})=0$. 
Lemma \ref{lem:11} is an easy consequence of (\ref{tag:11}) in view of $\prod_p\lam_{p,d^{-2}D_h}(p^{1/2})=(d^{-1}\frkf_h)^{1/2}$. 
\end{proof}

Our main result is the following: 
\begin{theorem}\label{thm:12}
If $F\in S_{k+n}(\Sp_{2n}(\ZZ))$ admits a Fourier expansion of the form (\ref{tag:13}),   
then there is $g\in S_{k+1/2}^+(4)$ with the following properties: 
\begin{enumerate}
\renewcommand\labelenumi{(\theenumi)}
\item If $n\equiv 0,1,3\pmod 4$, then $c(m)=c_g(m)$ for every $m\in\frkD_k$.     
\item If $n\equiv 2\pmod 4$, then  
\begin{enumerate}
\renewcommand\labelenumi{(\theenumi)} 
\item $F=I_{n,k}(g)$; 
\item If $m$ is not a square, then $c(m)=c_g(m)$; 
\item $c(f^2)=c_g(f^2)+(c(1)-c_g(1))f^k$ for every positive integer $f$. 
\end{enumerate}
\end{enumerate}
\end{theorem}

\begin{corollary}\label{cor:11}
The space $S^M_{k+n}(\Sp_{2n}(\ZZ))$ coincides with the image of the linear map $I_{n,k}$. 
\end{corollary}


\section{\bf Modular forms of half-integral weight}\label{sec:2}

We recall certain properties of newforms for the Kohnen plus space (see \cite{UY} for detail). 

Fix a positive integer $N$ such that $4^{-1}N$ is square-free. 
The space of cusp forms of weight $k+1/2$ with respect to $\Gam_0(N)$ is denoted by $S_{k+1/2}(N)$. 
We denote the $m$-th Fourier coefficient of $h\in S_{k+1/2}(N)$ by $c_h(m)$. 
The Kohnen plus space $S^+_{k+1/2}(N)$ consists of all functions $g\in S_{k+1/2}(N)$ such that $c_g(m)=0$ unless $m\in\frkD_k$. 

An operator $U_k(a^2)$ defined by 
\[\sum_{m\in\frkD_k}c(m)q^m|U_k(a^2)=\sum_{m\in\frkD_k}c(a^2m)q^m \]
maps $S_{k+1/2}^+(N)$ into itself 
if each prime divisor of $a$ is that of $4^{-1}N$. 

The space of newforms $S^{\new,+}_{k+1/2}(N)$ for $S^+_{k+1/2}(N)$ is the orthogonal complement of 
\[\sum_{p|4^{-1}N}\Big(S^+_{k+1/2}(p^{-1}N)+S^+_{k+1/2}(p^{-1}N)|U_k(p^2)\Big), \]
where $p$ extends over all prime divisors of $4^{-1}N$, in $S_{k+1/2}^+(N)$ with respect to the Petersson inner product. 
The space of newforms for $S_{2k}(\Gam_0(4^{-1}N))$ in the sense of Atkin-Lehner is denoted by $S^\new_{2k}(4^{-1}N)$. 
Let $\til{T}(p^2)$ (resp. $T(p)$) be the usual Hecke operator on the space of modular forms of half-integral (resp. integral) weight. 

The reader who wishes to learn the following theorem and proposition can consult \cite{UY}. 

\begin{theorem}\label{thm:21}

\begin{enumerate}
\renewcommand\labelenumi{(\theenumi)}
\item $S^+_{k+1/2}(N)=\oplus_{a,d\geq 1,\; ad|4^{-1}N}S^{\new,+}_{k+1/2}(4a)|U_k(d^2)$. 
\item There is a bijective correspondence, up to scalar multiple, between Hecke eigenforms in $S^\new_{2k}(4^{-1}N)$ and those in $S_{k+1/2}^{\new,+}(N)$ in the following way. 
If $f\in S^\new_{2k}(4^{-1}N)$ is a primitive form, i.e.,  
\begin{align*}
f|T(p)&=\ome_pf, & f|U(q)&=\ome_qf
\end{align*} 
for every prime number $p\nmid 4^{-1}N$ and prime divisor $q$ of $4^{-1}N$, then there is a non-zero Hecke eigenform $g\in S_{k+1/2}^{\new,+}(N)$ such that     
\begin{align*}
g|\til{T}(p^2)&=\ome_pg, & g|U_k(q^2)&=\ome_qg
\end{align*} 
for every prime number $p\nmid 4^{-1}N$ and prime divisor $q$ of $4^{-1}N$. 
\end{enumerate}
\end{theorem}

\begin{proposition}\label{prop:21}
Let $g\in S_{k+1/2}^{\new,+}(N)$, $p$ a prime divisor of $4^{-1}N$, and $\eps$ either $1$ or $-1$. 
Then the following conditions are equivalent: 
\begin{enumerate}
\renewcommand\labelenumi{(\theenumi)}
\item[(\roman{one})] if $\Psi_p((-1)^kn)=\eps$, then $c_g(n)=0$; 
\item[(\roman{two})] $g|U_k(p^2)=\eps p^{k-1}g$. 
\end{enumerate}
\end{proposition}


\section{\bf Results of \cite{Y}}\label{sec:3}

It is useful to recall certain properties of Jacobi forms before we turn to the proof of Theorem \ref{thm:12} (see \cite{Y} for detail). 

Let $S$ be a positive definite symmetric even integral matrix of rank $2n-1$. 
Put $L=\ZZ^{2n-1}$ and $L^*=S^{-1}L$. Let $\calt_S^+$ be the set of all pairs $(a, \alp)\in \ZZ\times L^*$ such that $a>S[\alp]/2$. 

A Jacobi cusp form $\vPh$ of weight $\kap$ with index $S$ is a holomorphic function on $\frkH\times\CC^{2n-1}$ which satisfies the following conditions:
\begin{enumerate}
\renewcommand\labelenumi{(\theenumi)}
\item[(\roman{one})] for every $\alp=\Left a & b \\ c & d\Right\in\SL_2(\ZZ)$, we have
\[\vPh(\alp\tau,w(c\tau+d)^{-1})=(c\tau+d)^\kap\bfe(c(c\tau+d)^{-1}S[w]/2)\vPh(\tau,w); \]
\item[(\roman{two})] for every $\xi$, $\eta\in L$, we have 
\[\vPh(\tau,w+\xi\cdot\tau+\eta)=\bfe(-S(\xi,w)-S[\xi]\cdot\tau/2)\vPh(\tau,w); \]
\item[(\roman{thr})] $\vPh$ has a Fourier expansion of the form 
\[\vPh(\tau, w)=\sum_{(a, \alp)\in\calt_S^+}c_\vPh(a, \alp)q^a\bfe(S(\alp, w)). \]
\end{enumerate}
Let $J^\cusp_{\kap,S}$ be the space of Jacobi cusp forms of weight $\kap$ with index $S$. 

For a place $v$ of $\QQ$, we put  
\[\eta_v(S)=h_v(S)(\det S, (-1)^{n-1}\det S)_v(-1,-1)_v^{n(n-1)/2}. \]
Here, let $(\;,\;)_v$ be the Hilbert symbol over $\QQ_v$ and $h_v$ the Hasse invariant (for the definition of the Hasse invariant, see \cite{Ki}). 

\begin{remark}\label{rem:31}
As is well-known, 
\[n=2\nu_p+2-\eta_p(S), \]
where $\nu_p$ is the Witt index of $S$ over $\QQ_p$ for each prime number $p$. 
\end{remark}

Let $q_p$ be the quadratic form on $V_p=L/pL$ defined by 
\[q_p[x]=S[x]/2\pmod p. \] 
We denote by $s_p(S)$ the dimension of the radical of $(V_p,q_p)$. 

We hereafter suppose that $L$ is a maximal integral lattice with respect to $S$. Then $s_p(S)\leq 2$. 
Put $\frkS_i=\{p\;|\;s_p(S)=i\}$. 
Let $b_S$ and $d_S$ be the product of rational primes in $\frkS_1$ and in $\frkS_2$ respectively. 
For a non-zero element $N\in\QQ$, the Laurent polynomial $l_{p,S,N}$ is defined by 
\[l_{p,S,N}=\begin{cases}
\lam_{p,N}
&\tx{if $p\in\frkS_0$. }\\
\lam_{p,N}+\eta_p(S)p^{1/2}\lam_{p,p^{-2}N}
&\tx{if $p\in\frkS_1$. }\\
\lam_{p,N}-\biggl(\dfrac{(-1)^np^{-2}N}{p}\biggl)p^{1/2}\lam_{p,p^{-2}N}-p\lam_{p,p^{-4}N}
&\tx{if $p\in\frkS_2$. }
\end{cases}\]
\begin{definition}
The space $J^{\cusp,M}_{\kap,S}$ is defined in the following way: 
$\vPh\in J^\cusp_{\kap,S}$ is an element of $J^{\cusp,M}_{\kap,S}$ if there exists a function $C:\frkD_k\to\CC$ such that $c_\vPh(a,\alp)=C((a-S[\alp]/2)\det S)$ for every $(a,\alp)\in\calt_S^+$. 
\end{definition}

The reader who wishes to learn the following theorem and proposition can consult \cite{Y}. 

\begin{theorem}\label{thm:31}
Let $k$ be a positive integer with $k\equiv n\pmod 2$. 
Let $b$ (resp. $d$) be a positive divisor of $b_S$ (resp. $d_S$). 
Let $f\in S^\new_{2k}(bd)$ be a primitive form, the $L$-function of which is given by
\[\prod_{p|bd}(1-\alp_pp^{k-1/2-s})^{-1}\prod_{p\nmid bd}(1-\alp_pp^{k-1/2-s})^{-1}(1-\alp_p^{-1}p^{k-1/2-s})^{-1}. \]
Suppose that $\alp_p=-\eta_p(S)p^{-1/2}$ for each prime divisor $p$ of $b$. Let $g\in S^+_{k+1/2}(4bd)$ be a corresponding cusp form under the Shimura correspondence. 
We define a function $\vPh$ on $\frkH\times\CC^{2n-1}$ by
\[\vPh(\tau,w)=\sum_{(a,\alp)}c_\vPh(a,\alp)q^a\bfe(S(\alp,w)), \]
where 
\[c_\vPh(a,\alp)=2^{-\bfb_{bd}(\pal_{a,\alp})}c_g(\frkd_{\pal_{a,\alp}})\frkf_{\pal_{a,\alp}}^{k-1/2}\prod_pl_{p,S,\pal_{a,\alp}}(\alp_p) \]
and $\pal_{a,\alp}=(a-S[\alp]/2)\det S$. 
Here, the number of prime divisors of $bd$ such that $\Psi_p((-1)^kN)\neq 0$ is denoted by $\bfb_{bd}(N)$ for a positive integer $N$. 
Then $\vPh$ is a cuspidal Hecke eigenform in $J^{\cusp,M}_{k+n,S}$. 
Moreover, the space $J^{\cusp,M}_{k+n,S}$ is spanned by these Jacobi cusp forms when $b$ and $d$ run through all positive divisors of $b_S$ and of $d_S$ respectively. 
\end{theorem}

The $S/2$-th Fourier-Jacobi coefficient $F_{S/2}$ of $F\in S_\kap(\Sp_{2n}(\ZZ))$ is defined by
\[F_{S/2}(\tau,w)=\sum_{(a,\alp)\in\calt_S^+}
c_F\Big(2^{-1}\mtrx{S}{S\alp}{\trs\alp S}{2a}\Big)q^a\bfe(S(\alp,w)). \]
As is well-known, $F_{S/2}$ is an element of $J^\cusp_{\kap,S}$. 

\begin{proposition}\label{prop:31}
Under the notation as in Theorem \ref{thm:31}, $\vPh$ is equal to the $S/2$-th Fourier-Jacobi coefficient of the Ikeda lift of $f$ if $b=d=1$. 
\end{proposition}


\section{\bf Proof of Main Theorem}\label{sec:4}

\begin{lemma}\label{lem:41}
Let $B\in T_{2n-1}^+$. Assume that $D_B$ is square-free. Then
\[\prod_{p|D_B}\eta_p(B)=\begin{cases}
1&\tx{if $n\equiv 0,1\pmod 4$. }\\
-1&\tx{if $n\equiv 2,3\pmod 4$. }
\end{cases}\]
\end{lemma}
\begin{proof}
By the definition of $\eta_p(B)$, we have 
\beq
\eta_\infty(B)=\begin{cases}
1&\tx{if $n\equiv 0,1\pmod 4$. }\\
-1&\tx{if $n\equiv 2,3\pmod 4$. }
\end{cases}\label{tag:41}
\eeq
If $p$ and $2D_B$ are coprime, then we can easily see that $\eta_p(B)=1$. 
In view of Remark \ref{rem:31}, it follows from the classification of maximal $\ZZ_2$-integral lattice that $D_B$ is even if $\eta_2(B)=-1$.  
The statement therefore follows from the Hasse principle. 
\end{proof}

\begin{lemma}\label{lem:42}
For an arbitrary rational prime $p$, there exists a positive definite symmetric half-integral matrix $B$ of degree $2n-1$ with $D_B=p$. 
\end{lemma}

\begin{proof}
It suffices to prove that there exists a positive definite even integral lattice $L$ with discriminant $2p$. 
We equip a lattice $M=\ZZ$ with the quadratic form $x\mapsto 2px^2$. This lattice is an even integral lattice with discriminant $2p$.   
Let $H\Pr$ be a definite quaternion algebra over $\QQ$ with discriminant $p$ and fix a maximal order $R\Pr$. 
Take a lattice $L\Pr$ such that $L\Pr\otimes_\ZZ\ZZ_p={\frkP\Pr}^{-1}$, where let $\frkP\Pr$ be the maximal ideal of $R\Pr\otimes_\ZZ\ZZ_p$, and $L\Pr\otimes_\ZZ\ZZ_q=R\Pr\otimes_\ZZ\ZZ_q$ for every rational prime $q$ distinct from $p$. 
We equip the lattice $M\Pr=\{x\in L\Pr\;|\;x\Iot=-x\}$ with the quadratic form $x\mapsto 2pxx\Iot$. Then this is an even integral lattice with discriminant $2p$. 

For each prime $q$, we define a lattice $L_q$ over $\ZZ_q$ as follows:
\[L_q=\begin{cases}
(M\otimes_\ZZ\ZZ_q)\oplus H_q^{n-1} &\tx{if $n\equiv 0,1\pmod 4$}\\
(M\Pr\otimes_\ZZ\ZZ_q)\oplus H_q^{n-2} &\tx{if $n\equiv 2,3\pmod 4$}
\end{cases}\]
(cf. (\ref{tag:41})). 
Here $H_q\simeq \ZZ_q^2$ is a hyperbolic plane over $\ZZ_q$. 

The Minkowski-Hasse theorem (see \cite[Theorem 4.1.2]{Ki}) shows that there exists a positive definite quadratic space $V$ over $\QQ$ such that $V\otimes_\QQ\QQ_q\simeq L_q\otimes_{\ZZ_q}\QQ_q$ for every rational prime $q$. There exists a lattice $L\subset V$ such that $L\otimes_\ZZ\ZZ_q\simeq L_q$ for every rational prime $q$.  
\end{proof}

Put $B_{a,\alp}=\Left B & B\alp \\ \trs \alp B & a\Right$ for $B\in T^+_{2n-1}$ and $(a,\alp)\in \calt_{2B}^+$. 

\begin{lemma}\label{lem:43}
Let $B\in T_{2n-1}^+$ and $m$ a positive integer.  
Assume that $D_B$ equals a rational prime $p$. 
Then the following conditions are equivalent: 
\begin{enumerate}
\renewcommand\labelenumi{(\theenumi)}
\item[(\roman{one})] there exists $(a,\alp)\in\calt^+_{2B}$ such that $D_{B_{a,\alp}}=m$;   
\item[(\roman{two})] $(-1)^nm\equiv0,1\pmod 4$ and
\[\biggl(\frac{(-1)^nm}{p}\biggl)\neq 
\begin{cases} 
-1 &\tx{if $n\equiv 0,1\pmod 4$. }\\ 
1 &\tx{if $n\equiv 2,3\pmod 4$. }
\end{cases}\]
\end{enumerate}
\end{lemma}

\begin{proof}
We have $\eta_p(B)=-1$ by virtue of Lemma \ref{lem:41}. 
Therefore Lemma \ref{lem:43} is a special case of the result in \cite[\S 5]{Y}. 
\end{proof}

The following corollary is a consequence of Lemma \ref{lem:42} and \ref{lem:43}. 
\begin{corollary}\label{cor:41}
We have
\[\{D_h\;|\;h\in T^+_{2n}\}=\begin{cases}
\frkD_n &\tx{if $n\equiv 0,1,3\pmod 4$. }\\
\frkD_n-\{1\} &\tx{if $n\equiv 2\pmod 4$. }
\end{cases}\]
\end{corollary}

We define the space $S^{\new,+,+}_{k+1/2}(4p)$ by  
\[S^{\new,+,+}_{k+1/2}(4p)=\{g\in S^{\new,+}_{k+1/2}(4p)\;|\; g|U_k(p^2)=p^{k-1}g\}. \] 

\begin{proposition}\label{prop:41}
Assume that $n\equiv 2,3\pmod 4$. 
Let $p$ be a rational prime and $B$ a positive definite symmetric half-integral matrix of degree $2n-1$ with $D_B=p$. 
Let $g\in S^+_{k+1/2}(4)$ and $h\in S^{\new,+,+}_{k+1/2}(4p)$. 
Put 
\[\vPh^B_{g,h}(\tau,w)=\sum_{(a,\alp)\in\calt_{2B}^+}c_{g,h}(a,\alp)q^a\bfe(2B(\alp,w)), \]
where 
\begin{multline*}
c_{g,h}(a,\alp)=c_g(D(B_{a,\alp}))-p^kc_g(p^{-2}D(B_{a,\alp}))+2^{-\bfb_p(D(B_{a,\alp}))}c_h(D(B_{a,\alp}))\\
\times p^{\frkf_p(D(B_{a,\alp}))/2}
(\lam_{p, D(B_{a,\alp})}(p^{-1/2})-p^{1/2}\lam_{p,p^{-2}D(B_{a,\alp})}(p^{-1/2})).  
\end{multline*}
Here, we set $D(B_{a,\alp})=D_{B_{a,\alp}}$. 
Then the linear map $(g,h)\mapsto\vPh^B_{g,h}$ maps $S^+_{k+1/2}(4)\oplus S^{\new,+,+}_{k+1/2}(4p)$ onto $J^{\cusp,M}_{k+n,2B}$. 
\end{proposition}

\begin{remark}\label{rem:51}
We can see that $J^{\cusp,M}_{k+n,2B}=J^\cusp_{k+n,2B}$ in this case. 
It follows from the proof of Theorem \ref{thm:31} that the map above is an isomorphism. 
These facts are not needed in the sequel. 
\end{remark}

\begin{proof}
Lemma \ref{lem:41} shows that $\eta_p(B)=-1$ and hence the polynomial $l_{q,2B,N}$ specifies as follows:
\[l_{q,2B,N}=\begin{cases}
\lam_{q,N} &\tx{if $q\neq p$. }\\
\lam_{p,N}-p^{1/2}\lam_{p,p^{-2}N} &\tx{if $q=p$. }
\end{cases}\]
If $g$ and $h$ are Hecke eigenforms, and $\{\alp_q,\alp_q^{-1}\}$ and $\{\bet_q,\bet_q^{-1}\}_{q\neq p}$ are the Satake parameters of elliptic cusp forms corresponding to $g$ and $h$ respectively, then Theorem \ref{thm:21} shows that 
\begin{gather}
c_g(m)=c_g(\frkd_m)\frkf_m^{k-1/2}\prod_q\lam_{q,m}(\alp_q), \label{tag:42}\\
c_h(m)=c_h(\frkd_m)\frkf_m^{k-1/2}p^{-\frkf_p(m)/2}\prod_{q\neq p}\lam_{q,m}(\bet_q). \label{tag:43}
\end{gather}
We thus see that $\vPh^B_{g,h}$ is a sum of Jacobi forms associated with $g$ and $h$ in Theorem \ref{thm:31}. 
Our statement is now immediate.  
\end{proof}

\begin{lemma}\label{lem:44}
Let $p$ and $B$ be the ones in Proposition \ref{prop:41} and Lemma \ref{lem:44}. If $k\equiv n\pmod 2$ and $F\in S_{k+n}(\Sp_{2n}(\ZZ))$ admits a Fourier expansion of the form (\ref{tag:13}), then 
\[c_F(B_{a,\alp})=c(D_{B_{a,\alp}})-p^kc(p^{-2}D_{B_{a,\alp}}). \]
\end{lemma}

\begin{proof}
Our assertion is a consequence of Proposition \ref{prop:31} and (\ref{tag:42}). 
This is also a special case of the result in \cite[\S 12]{Y}
\end{proof}

\begin{lemma}\label{lem:45}
Under the notation as in Lemma \ref{lem:44}, there exists a cusp form $g_B\in S_{k+1/2}^+(4)$ such that $F_B=\vPh^B_{g_B,0}$. 
\end{lemma}

\begin{proof}
Proposition \ref{prop:41} yields $g_B\in S^+_{k+1/2}(4)$ and $h_B\in S^{\new,+,+}_{k+1/2}(4p)$ such that $F_B=\vPh^B_{g_B,h_B}$. 

Seeking a contradiction, we assume that $h_B\neq 0$. 
Take $m\in \frkD_k$ such that $c_{h_B}(m)\neq 0$. 
By virtue of Proposition \ref{prop:21}, we see that $\Psi_p((-1)^km)\neq 1$ and can assume that $\frkf_p(m)=0$. 

For each non-negative integer $f$, Lemma \ref{lem:43} gives a pair $(a,\alp)\in\calt_{2B}^+$ such that $D_{B_{a,\alp}}=mp^{2f}$, and from (\ref{tag:43}) 
\[c_{h_B}(D(B_{a,\alp}))p^{\frkf_p(D(B_{a,\alp}))/2}=c_{h_B}(m)p^{(k-1/2)f}. \]
Thus Lemma \ref{lem:44} shows that 
\[c(mp^{2f})=c_{g_B}(mp^{2f})+2^{-\bfb_p(m)}c_{h_B}(m)p^{(k-1/2)f}\lam_{p,mp^{2f}}(p^{-1/2}). \]
Since the Ramanujan-Petersson conjecture suggests that $\alp_q$ and $\bet_q$ are of modulus $1$, (\ref{tag:42}) and (\ref{tag:43}) show that  
\begin{align*}
c_{g_B}(mq^{2f})&=O(q^{(k-1/2+\eps)f}), &
c_{h_B}(mq^{2f})&=O(q^{(k-1/2+\eps)f})
\end{align*}
for each prime $q$ and every $\eps>0$. Hence $c(mq^{2f})=O(q^{(k-1/2+\eps)f})$ 
for all rational primes $q$ distinct from $p$. 
We can observe that this is a contradiction, replacing $p$ with another rational prime. 
\end{proof}

\begin{lemma}\label{lem:46}
Let $q$ be a rational prime and $B\Pr$ a positive definite symmetric half-integral matrix of degree $2n-1$ with $D_{B\Pr}=q$. Under the notation as in Lemma \ref{lem:45}, we then have $g_B=g_{B\Pr}$.  
\end{lemma}

\begin{proof}
We define the map $P(\ell): S^+_{k+1/2}(4)\to S^+_{k+1/2}(4\ell)$ by 
\[g|P(\ell)=g-\ell^{1-k}g|(\til{T}(\ell^2)-U_k(\ell^2)) \]
for each prime number $\ell$. According to Lemma \ref{lem:43}, \ref{lem:44} and \ref{lem:45},   
\begin{align*}
g_B|P(p)(\tau)&=\sum_m\biggl(1-\biggl(\frac{(-1)^km}{p}\biggl)\biggl)(c_{g_B}(m)-p^kc_{g_B}(p^{-2}m))q^m \\
&=\sum_m\biggl(1-\biggl(\frac{(-1)^km}{p}\biggl)\biggl)(c(m)-p^kc(p^{-2}m))q^m
\end{align*}
and hence $g_B|P(p)P(q)=g_{B\Pr}|P(p)P(q)$. 
Operators $P(p)$ and $P(q)$ are injective in view of Proposition \ref{prop:21} (1).   
\end{proof}

Next is the proof of Theorem \ref{thm:12}. We can assume that $n\equiv 2,3\pmod 4$ by \cite{KK}. 
We define $g$ to be the function $g_B$ in Lemma \ref{lem:45}, which is independent of the choice of $B$ on account of Lemma \ref{lem:46}. 

Fix a positive integer $m$ such that $(-1)^km\equiv 0,1\pmod 4$. 
If there exists a rational prime $p$ such that 
\[\biggl(\frac{(-1)^km}{p}\biggl)\neq 1, \] 
then Lemma \ref{lem:42} and \ref{lem:43} yield $B\in T_{2n-1}^+$ and $(a,\alp)\in\calt_{2B}^+$ such that $D_B=p$ and $D_{B_{a,\alp}}=m$. 
Thus Lemma \ref{lem:44} and \ref{lem:45} show that 
\[c(m)-p^kc(p^{-2}m)=c_g(m)-p^kc_g(p^{-2}m). \]

If $(-1)^km$ is not a square, then we can find $p$ with $\Psi_p((-1)^km)\neq 1$ and deduce that $c(m)=c_g(m)$, 
which proves the case $n\equiv 3\pmod 4$. 

Assume that $n\equiv 2\pmod 4$. We then have 
\[c(f^2)=c_g(f^2)+(c(1)-c_g(1))f^k.\] 
In view of Lemma \ref{lem:11} and Remark \ref{rem:11}, the proof of Theorem \ref{thm:12} is now complete. 


\end{document}